\documentclass[smallextended]{svjour2}
\smartqed
\usepackage{mathptmx, amsmath, amsfonts, color,  epsfig}
\journalname{Numerische Mathematik}
\begin{document}
\title{Time-stepping discontinuous Galerkin methods for fractional diffusion  problems}
\titlerunning{DGMs for fractional diffusion problems}
\author{Kassem Mustapha
\thanks{The valuable comments of the editor and the referees improved the paper. The material of the paragraph: {\em Motivation of the $hp$-DG and future work}, on page 3 is mainly   based on my discussions with Professor Michael Griebel especially during my visit to Bonn on May 2014.
 The support of  the Science Technology Unit at KFUPM through  King Abdulaziz City for Science and
Technology (KACST) under National Science, Technology and Innovation Plan (NSTIP) project No. 13-MAT1847-04
 is gratefully acknowledged.}} \institute{  Kassem Mustapha \at Department of Mathematics and Statistics, King Fahd University of Petroleum
and Minerals, Dhahran,
 Saudi Arabia.\\
\email{kassem@kfupm.edu.sa} }
\date{}
\maketitle
\begin{abstract}
 Time-stepping $hp$-versions discontinuous Galerkin (DG)  methods for the numerical solution of fractional
subdiffusion problems of order $-\alpha$ with $-1<\alpha<0$ will be  proposed and analyzed. Generic $hp$-version error estimates are derived after proving the stability of
the approximate solution. For $h$-version DG approximations on appropriate graded meshes near~$t=0$, we prove that the error is of
order~$O(k^{\max\{2,p\}+\frac{\alpha}{2}})$, where $k$ is the maximum time-step size and  $p\ge 1$ is the uniform degree of the DG solution. For
$hp$-version DG approximations,  by employing geometrically refined time-steps and linearly increasing approximation orders, exponential rates
of convergence in the number of temporal degrees of freedom are shown.
 Finally, some numerical tests are given.
 \keywords{ Anomalous diffusion,  $hp$ methods, variable time steps,  error analysis}
\end{abstract}
\newcommand{\N}{{\mathcal N}}
\newcommand{\Poly}{\mathbb{P}}
\newcommand{\E}{{\mathcal E}}
\renewcommand{\H}{\mathbb{H}}
\newcommand{\A}{{\mathcal A}}
\newcommand{\I}{{\mathcal I}}
\newcommand{\X}{{\mathbb X}}
\newcommand{\G}{{\mathcal G}}
\newcommand{\Hilb}{\mathbb{H}}
\newcommand{\iprod}[1]{\langle#1\rangle}
\newcommand{\W}{{\mathcal W}}
\newcommand{\V}{{\mathcal V}}
\newcommand{\B}{{\mathcal B}}
\newcommand{\D}{{\mathcal D}}
\newcommand{\R}{{\mathbb R}}
\newcounter{bean}
\newenvironment{romenum}{\begin{list}{{(\roman{bean})}}
{\usecounter{bean}}}{\end{list}}
\section{Introduction}
In this work, time-stepping discontinuous Galerkin  methods (DGMs) for fractional order diffusion equations of the form
\begin{equation}
\label{eq:ivp}
u' +\B_\alpha Au = f\qquad~ {\rm on}~~ \Omega\times (0,T]\quad{\rm with}~~
u|_{t=0}=u_0,
\end{equation}
subject to homogeneous Dirichlet boundary conditions are proposed and analyzed, where $u=u(x,t)$, $f=f(x,t)$, $u_0=u_0(x)$ ($(x,t) \in \Omega\times  (0,T]$),  $u'=\frac{\partial u}{\partial t}$, and
\begin{equation} \label{eq:riemann-liouville} \B_\alpha v(t)=\frac{\partial}{\partial t}\int_0^t\omega_{\alpha+1} (t-s)\, v
(s)\,ds~~{\rm with}~~ {\omega_{\alpha+1}}(t)=\frac{t^{\alpha}}{\Gamma(\alpha+1)}
\end{equation}
is the Riemann--Liouville time fractional derivative  operator of order~$0<-\alpha<1$.

In \eqref{eq:ivp}, the spatial domain $\Omega$ is assumed to be bounded and polyhedral, and for simplicity,
 we choose  $A u
=-{\rm div} (K_{\alpha+1} \nabla u)$
  with $\nabla$ being the spatial gradient of $u$ and $K_{\alpha+1}>0$ (positive constant) is the diffusivity.
Thus,   $A$ (subject to homogeneous Dirichlet boundary conditions) is strictly positive-definite and possesses a complete orthonormal
eigensystem~$\{\phi_m\}_{m=1}^\infty$ in $L_2(\Omega)$. We let $\lambda_m$ denote the eigenvalue corresponding to~$\phi_m$ (i.e.,
$A\phi_m=\lambda_m\phi_m$ with $\phi_m|_{\partial \Omega}=0$) where (without loss of generality) we assume for convenience that  $0<
\lambda_1\le\lambda_2\le\lambda_3\le\cdots$.

Problems of the form~\eqref{eq:ivp} arise in a variety of physical, biological and chemical applications~\cite{KilbasSrivastavaTrujillo2006,Mathai2011,SokolovKlafter2005,Tarasov2011}.  It describes  anomalous subdiffusion and occurs, for example, in models of fractured or
porous media, where the particle flux depends on the entire history of the density gradient~$\nabla u$.

A variety of {\em low-order} numerical methods for problems of the form \eqref{eq:ivp} (with Riemann--Liouville or Gr\"unwald--Letnikov
fractional derivatives) were studied by several authors. For explicit, implicit Euler and  compact finite difference (FD) schemes, see for
example
\cite{ChenLiuTurnerAnh2010,Cui2009,LanglandsHenry2005,LiuYangBurrage2009,Mustapha2011,YusteAcedo2005,YusteQuintana2009,ZhuangLiuAnhTurner2008,ZhuangLiuAnhTurner2009}.
For ADI FD schemes  on a rectangular spatial domain; refer to  \cite{WangWang2011,ZhangSun2011}. In addition, various numerical methods
\cite{Cui2012,Cui2013,GaoSun2011,JinLazarovZhou2012,MustaphaAbdallahFurati2014,Quintana-MurilloYuste2011,SweilamKhaderMahdy2012,ZhangSun2011} have also been applied for the
following  alternative representation of (\ref{eq:ivp}) (using the Caputo derivatives):
$I^{-\alpha} u'(t) - {\rm div} (K_{\alpha+1} \nabla u)(t) = \tilde f(t)$ 
where $I^{-\alpha}$ is the Riemann--Liouville time fractional integral operator;
\[  I^{-\alpha} v(t):=  \int_0^t\omega_{-\alpha}(t-s)v(s)\,ds\quad{\rm for}~~-1<\alpha<0\,.
\]
 The two representations are equivalent
under suitable assumptions on the initial data, but the  methods obtained for each representation are formally
different.

In earlier papers, McLean and I proposed and analyzed  different {\em low-order}  time stepping DG schemes for problem \eqref{eq:ivp}. In  \cite{McLeanMustapha2009}, a {\em piecewise-constant}
  DGM (generalized backward Euler) combined with
 finite elements (FEs)  for the spatial discretization was studied. Unconditional stability and optimal convergence rates in both time and space were proved. Using a different approach, we later studied \cite{MustaphaMcLean2012IMA} the error analysis of the {\em piecewise-linear} DGM. Suboptimal rates of convergence had been achieved, however, the numerical results illustrated optimal rates.  In continuation, by duality arguments, nodal superconvergence results were proved in  \cite{MustaphaMcLean2012SINUM}. Moreover, we extracted the superconvergence at the nodal points of the DG solution globally by  post-processing the  DG solution  through
 Lagrange interpolations.   In all these papers, variable time steps were employed to compensate the lack of regularity of the solution $u$ of problem \eqref{eq:ivp} near $t=0$.

The main purpose of this paper is to study the stability and the accuracy of  {\em high-order}  time-stepping $h$-version DG ($h$-DG) and $hp$-version ($hp$-DG)  methods. This task is not trivial since the DGM allows us to only control the jumps of the approximate solutions which is enough for the {\em low-order} DGMs in \cite{McLeanMustapha2009,MustaphaMcLean2011,MustaphaMcLean2012IMA,MustaphaMcLean2012SINUM}. A new analysis based on the  coercivity and continuity properties of the operator $\B_\alpha$, also based on some fractional derivative-integral  identities is required. These will be the keys to establish the stability and consequently  deriving  promising error estimates (in the $L_\infty(0,T)$-norm)  over families of nonuniform  meshes.

  In contrast, for $0<\alpha<1,$ the model problem \eqref{eq:ivp} amounts to the fractional wave equation (super-diffusion):
  \begin{equation}
\label{eq: fractional wave}
u'(t) +I^{\alpha} Au (t)  = f(t)\quad{\rm with}~~
u|_{t=0}=u_0\,.
\end{equation}
Recently, Sch{\"o}tzau and I investigated
\cite{MustaphaSchoetzau2013}     $h$-DG and  $hp$-DG methods (in time)
  for problem \eqref{eq: fractional wave}. Although we could not show the stability of our scheme  because of some technical difficulties, algebraic and exponential convergence rates (in a non-standard norm that can be weaker than the $L_2$-norm in some cases) for $h$-DG and $hp$-DG schemes were achieved, respectively.

  Due to the different nature and properties of the  operators $I^\alpha$ and  $\B_{\alpha}$, another technique will be used in this work  to show the stability of our scheme and also to derive generic $hp$-version abstract error estimates in the  stronger $L_\infty(0,T)$-norm.
Then proceeding along the lines of  \cite{MustaphaBrunnerMustaphaSchoetzau2011,MustaphaSchoetzau2013,SchoetzauSchwab00}  and investigating two refinement strategies in the case where the solution $u$ of \eqref{eq:ivp} lacks regularity as $t=0$. Noting that, in \cite{MustaphaBrunnerMustaphaSchoetzau2011,SchoetzauSchwab00},  $hp$-DGMs   for parabolic and parabolic integro-differential equations were considered   where the stability and error analyses  follow relatively straight forwardly from the different natures of the equations.

 In the $L_\infty(0,T)$-norm,  exponential convergence rates (in the number of temporal degrees of freedom) for the  $hp$-version DG ($hp$-DG) scheme based on geometrically refined time-steps and on linearly increasing approximation orders will be achieved. Moreover, for the $h$-DGM of piecewise uniform  degree $p\ge 1$, we prove $O(k^{\max\{2,p+\frac{1}{2}\}+\frac{\alpha}{2}})$  algebraic convergence rates over non-uniform graded meshes that concentrate the time levels near~$t=0$.  So, the convergence rates is short $\frac{1-\alpha}{2}$ power from being optimal for $p \ge 2$, however, just short by $-\frac{\alpha}{2}$ power  for $p=1$ which is due to the fact that  the approximate solution can be controlled from the jumps in this case. Indeed, the numerical experiments illustrate optimal error rates of order~$O(k^{p+1})$ for some choices of $\alpha$ and $p$.  In our test we combine the proposed $hp$-version time-stepping method with a standard (continuous) FEs in space which will then define a fully discrete scheme. We choose the spatial step size and the order of the spatial FEs so that the temporal errors are dominating. Analyzing the convergence of the fully discrete  scheme will be considered in future work.

{\em Motivation of the $hp$-DG and future work.} The nonlocal nature of $\B_\alpha$ means that on each time subinterval, one must efficiently evaluate a sum of integrals over all previous time subintervals. For example, a direct implementation of the time-stepping $h$-DG method (of uniform degree $p$) combined with the FE discretization in space requires $\mathcal{O}((p+1)\,N^2\,M)$ operations and requires $\mathcal{O}((p+1)\,N\,M)$ storage ($N$ is the number of time-mesh elements and $M$ is the spatial degrees of freedom). Thus, reducing the number of time-steps and at the same time maintaining  high accuracy is important especially when $\Omega \subset\R^3$, then the time-space problem \eqref{eq:ivp} is four-dimensional and thus beyond the computing power of conventional machines.  For analytic solutions in the time variable, $hp$-DGMs with exponential rates of convergence allow us to achieve these requirements  to a large extent. For instance, if the error from the spatial FE is of order $\mathcal{O}(h^r)$ for some $r\ge 2,$ then we can balance the exponential rates in time with the algebraic one in space. In this case, the number of operations will be reduced to $\mathcal{O}((p+1)\,M\,(\log M)^{\nu_2})$ and the active operations to  $\mathcal{O}((p+1)\,M(\log M)^{\nu_2})$ where $\nu_1$ and $\nu_2$ depend on $r$. However, if the solution $u$ of \eqref{eq:ivp} is not analytic (in time) but satisfies appropriate  regularity assumptions,  time-space sparse grids can be used to get similar  results. Furthermore, if in addition $u$ satisfies certain mixed spatial regularity properties,  the computational cost can be  further reduced. These
computing issues are subject to ongoing  investigation and hence, will be considered in future work.

The outline of the paper is as follows. The  DGM will be introduced in the next section and the stability of the semi-discrete solutions
will be proved in Section~\ref{sec:existence-uniqueness}. Followed by deriving abstract error bounds of the time-stepping DGM  in
Section~\ref{sec:error}. Sections~\ref{sec:algebraic} and \ref{sec:exponential}  are devoted to establishing algebraic rate of convergence of
the $h$-DGM and exponential rates of convergence for the $hp$-DGM, respectively. Numerical illustrations of our results will be presented in
Section~\ref{sec:numerics}.
\section{Discontinuous Galerkin discretization}
\label{sec:numerical-method}
To define the time-stepping DGM for problem \eqref{eq:ivp}, we introduce a  partition~${\cal M}$ of the interval
$[0,T]$ given by the points:
$0=t_0<t_1<\cdots<t_N=T.$
We set $I_n=(t_{n-1},t_n)$ and $k_n=t_n-t_{n-1}$ for $1\le n\le N$.
With each
subinterval $I_n$ we associate a polynomial degree $p_n\in{\mathbb N}_0$. These degrees are then stored in the degree vector ${\bf
p}:=(p_1,p_2,\cdots,p_N).$ Next, we introduce the discontinuous finite element space
\begin{equation}
\label{eq:FE-space} \W({\cal M}, {\bf p})=\left\{\,v:[0,T]\to H^1_0(\Omega)\,:\, v|_{I_n}\in\Poly_{p_n},\ 1\le n\le N\right\},
\end{equation}
where $\Poly_{p_n}$ denotes the space of polynomials of degree $\le p_n$ with coefficients in~$H^1_0(\Omega)$. For a function~$v\in\W({\cal
M},{\bf p})$, we write $ v^n_-=v(t_n^-)$, $v_+^n=v(t_n^+)$ and $[v]^n=v^n_+-v^n_-$ with $ v^N_+=v^N_-$ and  $v_-^0=v^0\,.$

The time-stepping DG approximation $U\in \W({\cal M},{\bf p})$ is now defined as follows: Given $U(t)$ for $ t \in \cup_{j=1}^{n-1} I_{j-1}$, the discrete
solution $U\in \Poly_{p_n}$ on the next time subinterval ~$I_n$ is determined by requesting that
\begin{equation}
\label{eq: DG step}  \iprod{U^{n-1}_+-U^{n-1}_-,X^{n-1}_+}
    +\int_{t_{n-1}}^{t_n}\,\Big(\iprod{U',X}+A(\B_\alpha U,X)\,\Big)\,dt
=\int_{t_{n-1}}^{t_n}\iprod{f,X}\,dt
\end{equation}
$\forall ~X\in\Poly_{p_n}$ with ~$U^0_-\approx u_0$. Here   $A(\cdot,\cdot): H^1_0(\Omega)\times H^1_0(\Omega) \to\mathbb{R}$  is bilinear operator
associated with the differential operator~$A:=-{\rm div} (K_{\alpha+1} \nabla)$ { and is }given by
\begin{equation*}
A(v,w):=\iprod{\nabla v,\nabla w}=\sum_{m=1}^\infty\lambda_m\, u_m\,v_m \quad\text{where $u_m=\iprod{u,\phi_m}$ and $v_m=\iprod{v,\phi_m}$}\,.
\end{equation*}
Throughout the paper, by $\iprod{\cdot,\cdot}$  and $\|\cdot\|$, we denote the inner product and the associated norm  in the space
$L_2(\Omega)$. Moreover, $\|\cdot\|_1$ denotes the norm on the Sobolev space $H^1(\Omega)$ and for $j\ge 1$, $u^{(j)}:=\frac{\partial^j u}{\partial
t^j}$.

As in \cite{MustaphaSchoetzau2013}, since the  operator $A$ possesses a complete orthonormal eigensystem~$\{\lambda_m, \phi_m\}_{m\geq 1}$, the DG scheme ~\eqref{eq: DG step}
can be reduced to a finite linear system of algebraic equations on each subinterval $I_n$. To see this, let~$P_{p_n}$ be the scalar polynomial
space of degree $\le p_n$. Now, take  $X=\phi_m w$ in~\eqref{eq: DG step}, we find that: for $m\ge 1,$

\begin{equation}
\label{eq:DG-step-algebraic}  U_{m,+}^{n-1}\; w^{n-1}_+
    +\int_{t_{n-1}}^{t_n}\,
    \Bigl({U'_m}w+\lambda_m \B_\alpha U_m\; w  \,\Bigr)\,dt
\\=U^{n-1}_{m,-}\;w^{n-1}_++\int_{t_{n-1}}^{t_n}f_m\,w\,dt
\end{equation}
$\forall~ w \in P_{p_n}$ and for $1\le n\le N,$ where $U_m=\iprod{U,\phi_m}\in P_{p_n}$ and $f_m=\iprod{f,\phi_m}$.

Very briefly, because of the finite dimensionality of system~\eqref{eq:DG-step-algebraic}, the existence of the scalar
function $U_m$ on $I_n$  follows from its uniqueness. For uniqueness, it is enough to show that $U_m \equiv 0$ on $I_n$ for $n\ge 1$ when the
right-hand side of \eqref{eq:DG-step-algebraic} is identically zero. This follows from the stability theorem (Theorem \ref{theorem: stability})
and the coercivity property (i) in Lemma \ref{lemma: lower bound of the positivity}.

\section{Stability of DG solutions}
\label{sec:existence-uniqueness}
In this section, we show  the stability of the semi-discrete solutions.  For convenience, we introduce the following notation. Set $J:=\bigcup_{j=1}^N
I_j$ and  we let ${\mathcal C}^1(J)$ denote the space of functions~$v:J\to \mathbb{R}$ such that the restriction~$v|_{I_j}$ extends to a continuously
differentiable function on the closed interval~$[t_{j-1},t_j]$, for~$1\le j\le N$.

 In the following result we gather  two key
properties of the fractional time derivative operator $\B_\alpha$  that we use in our analysis.
 \begin{lemma} \label{lemma: lower bound of the positivity} Let $c_\alpha =\frac{ \cos(
\alpha\pi/2)}{\pi^\alpha}\frac{|\alpha|^{-\alpha}}{(1-\alpha)^{1-\alpha}}$ and $ d_\alpha=\frac{1}{\cos(\alpha \pi/2)}$ for  any $-1<\alpha<0$. Then,
 for any $v,\,w$ in  ${\mathcal C}^1(J)$ (or in $W^{1,1}(0,T)$), we have
\begin{itemize}
\item[{\rm (i)}]
$\int_0^T \B_\alpha v(t)\,v(t)\,dt \ge c_\alpha T^\alpha \int_0^T v^2(t)\,dt,$
\item[{\rm (ii)}]
$\Big| \int_0^T \B_\alpha v(t)\,w(t)\,dt\Big|^2
    \le {d_\alpha}^2 \int_0^T \B_\alpha v(t)\,v(t)\,dt\int_0^T \B_\alpha w(t)\,w(t)\,dt\,.$
\end{itemize}
\end{lemma}
\begin{proof}
The coercivity property (i) was proven in \cite[Theorem A.1]{McLean2012} by using the Laplace transform and the Plancherel Theorem. In a similar
fashion, property (ii) can be obtained, see for example \cite[Lemma 3.1]{MustaphaSchoetzau2013}.$\quad \Box$
\end{proof}
\begin{remark} Noting that, as $\alpha$ approaches  $0$, we recover the classical coercivity and continuity properties. In addition, as was mentioned earlier, for $0<\alpha<1,$ $\B_\alpha=I^\alpha$. In this case, the above coercivity property is no longer valid. We have  a weaker version instead, see \cite[Lemma 3.1 (i)]{MustaphaSchoetzau2013}.
$ \quad \Box$
\end{remark}
The stability of the DG solution $U$ will be shown in the next theorem.  The proof below looks straightforward due to the new approach that has not been used before. The key ingredients are the above lemma and the appropriate use of the identity:   $ \B_\alpha  I^{-\alpha}=I.$ Indeed, the current approach can be adopted to show the stability of $U$ when $0<\alpha<1$ as this was not proven  in \cite{MustaphaSchoetzau2013}.    It is worth mentioning that the stability result below plays a crucial role in our forthcoming error analysis, see Theorem \ref{lem: bound of theta}.   Noting that, the proofs of the stability in \cite{McLeanMustapha2009,MustaphaMcLean2011}  are valid only for $h$-DGMs of order $p\in \{0,1\}$ ({\em low-order}).
\begin{theorem}\label{theorem: stability}
For $1\leq n \leq N$, the DG solution $U$ of~\eqref{eq: DG step} satisfies
\[
\|U_-^n\|^2+\|U^{n-1}_+\|^2+2\,\int_0^{t_n}\, A(\B_\alpha U,U) \,dt
   \leq  4\,\|U_-^0\|^2
   +4\, d_\alpha^2 \int_0^{t_n} |\iprod{g,  A^{-1} f }|\,dt
\]
 where 
  $A^{-1}$ is the inverse of the
 positive-definite operator $A$, and $g(t):=(I^{-\alpha} f)(t)\,.$
\end{theorem}
\begin{proof}
Choosing $X=U$ in~\eqref{eq: DG step} and  using  $\iprod{U'(t),U(t)}=\frac{1}{2}\frac{d}{dt}\|U(t)\|^2$, we obtain
\[
\|U_-^j\|^2+\|U^{j-1}_+\|^2-2\iprod{U_-^{j-1},U^{j-1}_+}
    +2\int_{t_{j-1}}^{t_j}\, A(\B_\alpha U,U) \,dt
    =2\int_{t_{j-1}}^{t_j}\iprod{ f,U}\,dt.
\]
Summing over $j=1,\cdots,\ell$, and using  $f=\B_\alpha g,$
\begin{multline*}
 \sum_{j=1}^\ell\bigl(\|U_-^j\|^2+\|U^{j-1}_+\|^2-2\iprod{U_-^{j-1},U^{j-1}_+}\bigr)+2\int_0^{t_\ell}\, A(\B_\alpha U,U) \,dt\\
  =
    2\, \int_0^{t_\ell} \iprod{\B_\alpha  g,U}\,dt.
\end{multline*}
Choose $\ell=n-1$ and $\ell=n$ respectively, summing and then using  the identity;
\begin{multline*}\sum_{j=1}^\ell\bigl(\|U_-^j\|^2+\|U^{j-1}_+\|^2-2\iprod{U_-^{j-1},U^{j-1}_+}\bigr)\\
=\|U_-^\ell\|^2+\|U^0_+\|^2+\sum_{j=1}^{\ell-1}\|[U]^j\|^2-2\iprod{U_-^0,U^0_+}\end{multline*} yield
 \begin{multline*}
\|U_-^{n-1}\|^2+  \|U_-^n\|^2+2\|U^0_+\|^2+\|[U]^{n-1}\|^2+2\sum_{j=n-1}^n \int_0^{t_j}\, A(\B_\alpha U,U) \,dt\\
   \le  4\iprod{U_-^0,U^0_+}
    +2\sum_{j=n-1}^n  \int_0^{t_j}\iprod{\B_\alpha  g,U}\,dt.
\end{multline*}
Since $4|\iprod{U_-^0,U^0_+}|\le 2\|U_-^0\|^2+2\|U^0_+\|^2$ and  $\|U^{n-1}_+\|^2  \le 2\|U_-^{n-1}\|^2+2\|[U]^{n-1}\|^2$,  \begin{multline}
 \label{eq: U^2 bound}
\|U_-^n\|^2+\|U^{n-1}_+\|^2+4\sum_{j=n-1}^n\int_0^{t_j}\, A(\B_\alpha U,U) \,dt\\
   \leq  4\,\|U_-^0\|^2
   +4\sum_{j=n-1}^n  \int_0^{t_j} \iprod{ \B_\alpha  g,U}\,dt,
\end{multline}
for $1\le n\le N$. Now,  setting $U_m=\iprod{U,\phi_m}$ and $g_m=\iprod{g,\phi_m}$, and hence,    the continuity property (ii) in Lemma~\ref{lemma: lower bound of the positivity} implies that; for $1\le j\le N$,
\begin{align*}
4\, \int_0^{t_j}\iprod{\B_\alpha  g,U}\,dt&=4\sum_{m=1}^\infty\,\int_0^{t_j} \B_\alpha  g_m\,U_m\,dt\\
&\le 4\,d_\alpha \sum_{m=1}^\infty\,\left(\int_0^{t_j} \B_\alpha  g_m\,g_m\,dt\right)^{1/2}\left(\int_0^{t_j} \B_\alpha  U_m\,U_m\,dt\right)^{1/2}\,dt\\
&\le 2\, d_\alpha^2 \sum_{m=1}^\infty\,\int_0^{t_j} \B_\alpha  g_m\,\lambda_m^{-1} g_m\,dt+2\sum_{m=1}^\infty\lambda_m \int_0^{t_j} \B_\alpha
U_m\,U_m\,dt
\\
&= 2\, d_\alpha^2 \int_0^{t_j} \iprod{\B_\alpha g , A^{-1} g}\,dt+2 \, \int_0^{t_j} A(\B_\alpha U , U)\,dt\,.
 \end{align*}
Therefore,  the desired stability estimate follows after inserting the above bound (for $j=n-1$ and $j=n$)
 on the right-hand side of \eqref{eq: U^2 bound}\,.
This finishes the proof.$\quad \Box$
\end{proof}

\section{Error analysis}
\label{sec:error}
This section is devoted to deriving abstract error estimates for the DGM. A global formulation of our numerical scheme will be given first. More
precisely, it will be convenient to reformulate the DG scheme (\ref{eq: DG step}) in terms of the  bilinear form
\begin{equation}\label{eq: GN def}
G_N(U,X)=\iprod{U^0_+,X^0_+}+\sum_{n=1}^{N-1}\iprod{[U]^n,X^n_+}
    +\int_0^{t_N}\,\Bigl(\iprod{U',X}+A(\B_\alpha U,X)\Bigr)\,dt.
\end{equation}
Integration by parts yields an alternative  expression for the bilinear form $G_N$:
\begin{equation}
\label{rem:GN-alt}G_N(U,X)=\iprod{U^N_-,X^N_-}-\sum_{n=1}^{N-1}\iprod{U^n_-,[X]^n} +\int_0^{t_N}\,\Bigl(-\iprod{U,X'}
    + A(\B_\alpha U,X)\Bigr)\,dt.
\end{equation}
 By summing up (\ref{eq: DG step}) over all time-steps, the DGM can now equivalently be written as: Find $U\in \W({\cal M},{\bf p})$
such that
\begin{equation}\label{eq:DGFEM}
G_N(U,X)=\iprod{U^0_-,X^0_+}+\int_0^{t_N} \iprod{f,X}\,dt\qquad\forall\, X\in\W({\cal M},{\bf p}).
\end{equation}

 Let $u$ be the solution of~(\ref{eq:ivp}) and $U$ the DG approximation defined
in~(\ref{eq:DGFEM}).  Decomposing the error $U-u$ into the two terms:
\begin{equation}\label{eq: decompose U-u}
U-u=(U-\Pi u)+(\Pi u-u)=:\theta+ \eta\,.
\end{equation}
where $\Pi u \in \W({\cal M},{\bf p})$ is the  $hp$-version projection of $u$ defined by: for $1\le n\le N,$
\begin{equation}
\label{eq:Pi^p-def} \Pi u(t_n^-)-u(t_n)=0\quad\text{and}\quad \int_{t_{n-1}}^{t_n}\, \iprod{ u\,-\Pi u,v} \,dt=0\quad \forall\,v\in
\Poly_{p_n-1},
\end{equation}

The bound of $\eta$ follows from the next theorem.
\begin{theorem}\label{thm:eta-approx}
Let $1\leq n \leq N$ and $0\le q_n\le p_n$. If $u^{(q_n+1)}|_{I_n} \in L_2(I_n;H^1(\Omega))$, then
\[
 \int_{t_{n-1}}^{t_n} \|(\Pi u-u)'\|^2_1\,dt
\le C\,p_n^2\left(\frac{k_n}{2}\right)^{2q_n}\Gamma_{p_n,q_n}\int_{t_{n-1}}^{t_n} \|u^{(q_n+1)}\|^2_1\,dt
\]
where $\Gamma_{p_n,q_n}=\frac{\Gamma(p_n-q_n+1)}{\Gamma(p_n+q_n+1)}$  and the constant $C$ is independent of $k_n$, $p_n$, $q_n$, and $u$.
\end{theorem}
\begin{proof}  See ~\cite[Section~3]{SchoetzauSchwab00} for the proof.$\quad \Box$ \end{proof}
The main task now is to estimate ~$\theta$. To do so, we use the contribution from the stability results, the continuity property of the operator $\B_\alpha$, the inverse inequality, in addition to some other technical steps. In comparison to the case $0<\alpha<1$, the achieved bound of $\theta$ in \cite{MustaphaSchoetzau2013} is  weaker (by far) than the one below. This is due to the different properties of $\B_\alpha$ and $\I^\alpha$ and also because of the technique used here.

\begin{theorem}\label{lem: bound of theta} Assume that the time-step
 sizes are nondecreasing. Then, for $1\le n\le N,$ if the solution $u \in W^{1,1}((0,t_n); H^1_0(\Omega))$, we have
\begin{multline*}|\theta|_n^2+ \frac{k_n^{\min\{1,p_n-1\}}}{p_n^2}\Big( \sup_{t\in I_n} \|\theta(t)\|^2\Big) \\
\leq C_{\alpha,T} \left(\,\|U^0_--u_0\|^2 + t_n \max_{j=1}^{n}
  k_j^{\alpha}\,\big(\int_{t_{j-1}}^{t_j}\| \eta'\|_1 dt\big)^2\right),
\end{multline*}
where  $|\theta|_n:=\max\{\|\theta^{n-1}_+\|,\|\theta^n_-\|\}$.
\end{theorem}
\begin{proof}
Since  $G_n(u,X)=\iprod{u_0,X^0_+}+\int_0^{t_n} \iprod{f,X}\,dt$, 
we have
\[ G_n(U-u,X)=\iprod{U^0_--u_0,X^0_+}\qquad \forall\, X\in\W({\cal M},{\bf p}).\]
Hence,  the alternative expression for~$G_N$ in \eqref{rem:GN-alt}, and the fact that
$\eta^n=0$ and $\int_{t_{n-1}}^{t_n}\iprod{\eta,X'}\,dt=0$ for all $1\leq n\leq N$,  by definition of the operator $\Pi$ (note that for $p_n=0$,
we have $X'\equiv 0$), yield
\begin{align*}
G_n(\theta,X)&=\iprod{U_-^0-u_0,X^0_+}-G_n(\eta,X) =\iprod{U_-^0-u_0,X^0_+}+\int_0^{t_n}\,
     A(\B_\alpha \eta,X)\,dt
\end{align*}
for all $X\in\W({\cal M},{\bf p})$. Since this equation has the same form as~\eqref{eq:DGFEM}, following the proof of the stability in
Theorem~\ref{theorem: stability}, we notice that for $1\le n\le N$,
\begin{multline}
\label{eq:q1+q2+q3 new} |\theta|_n^2+2\,\int_0^{t_n} A(\B_\alpha \theta,\theta)\,dt
   \leq  4\,\|U^0_--u_0\|^2
   +4\, d_\alpha^2\,\max_{j=n-1}^n \int_0^{t_j} A(\B_\alpha \eta,\eta)\,dt
\,.
\end{multline}
To estimate the last term, we use the equality $\eta(t)=-\int_t^{t_j}\eta'(q)\,dq$ for $t\in I_j$,
 then  changing the order of integrations and integrating,
\begin{equation}\label{eq: mid-step}
\begin{aligned}
\int_0^{t_n} A(\B_\alpha \eta,\eta)\,dt&=
-\sum_{j=1}^n \int_{t_{j-1}}^{t_j} \int_t^{t_j}  A(\B_\alpha\eta(t),\eta'(q))\,dq\,dt\\
&= -\sum_{j=1}^n\int_{t_{j-1}}^{t_j}\int_{t_{j-1}}^q   A(\B_\alpha\eta(t),\eta'(q))\,dt\,dq\\
    & =  \sum_{j=1}^n\int_{t_{j-1}}^{t_j}A(\I^{-\alpha}\eta(t_{j-1})-\I^{-\alpha}\eta(q), \eta'(q))\,dq \\
    &= \sum_{j=1}^n\I_1^j+\sum_{j=1}^n\I_2^j
\end{aligned}
\end{equation}
where
\[\begin{split}
\I_1^j&:=-  \int_{t_{j-1}}^{t_j}\int_{t_{j-1}}^q
    \omega_{\alpha+1}(q-s)A(\eta(s),\eta'(q))\,ds\,dq\\
    &= \int_{t_{j-1}}^{t_j}\int_{t_{j-1}}^q\int_s^{t_j}
    \omega_{\alpha+1}(q-s)A(\eta'(t),\eta'(q))\,dt\,ds\,dq\end{split}\]
    and
    \[\begin{split} \I_2^j&:= \int_{t_{j-1}}^{t_j} \int_0^{t_{j-1}}
    [\omega_{\alpha+1}(t_{j-1}-s)-\omega_{\alpha+1}(q-s)] A(\eta(s), \eta'(q))\,ds\,dq\\
    &=-\sum_{i=1}^{j-1} \int_{t_{j-1}}^{t_j} \int_{t_{i-1}}^{t_i}\int_s^{t_i}
    [\omega_{\alpha+1}(t_{j-1}-s)-\omega_{\alpha+1}(q-s)] A(\eta'(t), \eta'(q))\,dt\,ds\,dq\,.\end{split}\]
 To bound $\I_1^j$ and $\I_2^j$, we use the Cauchy-Schwarz inequality and  integrating
 \begin{align*} \I_1^j &\le  \int_{t_{j-1}}^{t_j}\|\nabla \eta'(t)\|\int_{t_{j-1}}^q \omega_{\alpha+1}(q-s)\int_{t_{j-1}}^{t_j} \|\nabla \eta'(q)\|\,dt\,ds\,dq
\\ &\le  \omega_{2+\alpha}(k_j)\Big(\int_{t_{j-1}}^{t_j}\|\eta'\|_1 dt\Big)^2
\end{align*}
and
\begin{align*}
\I_2^j &\le \sum_{i=1}^{j-1}   \int_{t_{j-1}}^{t_j}\|\nabla \eta'(q)\| \int_{t_{i-1}}^{t_i}
    [\omega_{\alpha+1}(t_{j-1}-s)-\omega_{\alpha+1}(t_j-s)] \int_s^{t_i} \|\nabla \eta'(t)\| \,dt\,ds\,dq\\
&\le   \max_{i=1}^{j}\, \left(\int_{t_{i-1}}^{t_i}\|\nabla \eta'\|dt\right)^2 \, \int_0^{t_{j-1}}
    [\omega_{\alpha+1}(t_{j-1}-s)-\omega_{\alpha+1}(t_j-s)] \,ds
\\
    &\le    \omega_{2+\alpha}(k_j) \max_{i=1}^{j}  \Big(\int_{t_{i-1}}^{t_i}\| \eta'\|_1dt\Big)^2\,.
\end{align*}
Now, inserting  the estimates   of $\I_1^j$ and $\I_2^j$ in \eqref{eq: mid-step}, and using the mesh assumption $k_i\le k_j$ for $i\le j$. This
implies
\[
\int_0^{t_n} A(\B_\alpha \eta,\eta)\,dt \leq \frac{2\,t_n}{\Gamma(\alpha+2)} \max_{j=1}^{n}
 k_j^{\alpha}\,\Big(\int_{t_{j-1}}^{t_j}\| \eta'\|_1 dt\Big)^2,\]
and therefore, for $1\le n\le N$, \[|\theta|_n^2+2\,\int_0^{t_n} A(\B_\alpha \theta,\theta)\,dt
   \leq  4\,\|U^0_--u_0\|^2
  +\frac{4\,t_n \, d_\alpha^2}{\Gamma(\alpha+2)} \max_{j=1}^{n}  k_j^{\alpha}\Big(\int_{t_{j-1}}^{t_j}\| \eta'\|_1dt\Big)^2\,.
\]
But, for $p_n=1$, the left-hand side is $\ge \sup_{t\in I_n} \|\theta(t)\|^2$, however for $p_n\ge 2$, it is
\[\ge c_\alpha\,t_n^{\alpha}\int_{t_{n-1}}^{t_n}\|\nabla \theta\|^2\,dt
\ge C\,c_\alpha t_n^{\alpha}\int_{t_{n-1}}^{t_n}\|\theta\|^2\,dt\ge C\,c_\alpha t_n^{\alpha}\frac{k_n}{p_n^2}\Big(\sup_{t\in I_n}\|\theta(t)\|^2\Big)\] by the
assumption that the operator  $A$  possesses a complete orthonormal eigensystem~$\{\lambda_m, \phi_m\}_{m=1}^\infty$, the coercivity property in
Lemma \ref{lemma: lower bound of the positivity} (i), and the Poincare's ($\theta|_{\partial \Omega}=0$) and inverse ($\theta|_{I_n} \in
\Poly_{p_n}$) inequalities. This completes the proof.$\quad \Box$
\end{proof}

The main abstract error bound will be derived in the next theorem. For convenience, we  introduce the following notation:
\[\|\phi\|_{L_\infty(L_2)}:=\|\phi\|_{L_\infty((0,T);L_2(\Omega))}=\max_{n=1}^N\Big(\sup_{t\in I_n} \|\phi(t)\|\Big).\]
\begin{theorem}\label{thm:error-bound-abstract}
Let $u$ be the solution of (\ref{eq:ivp}) and $U$ be the DG solution defined by~\eqref{eq: DG step} with $U^0_-=u_0$ (for
simplicity). Assume that  $k_i\le k_j$ for $i\le j$. Then we have
\begin{multline*}\max_{ n=1}^ N\{\|U^{n-1}_+-u(t_{n-1})\|,\|U^{n}_--u(t_{n})\|\}
  \\+\min_{n=1}^{N}\left(t_n^{-1} k_n^{\min\{1,p_n-1\}}\right) \frac{\|u-U\|_{L_\infty(L_2)}^2}{\max_{n=1}^N p_n^2}
    \le C_{\alpha,T}\, \max_{n=1}^{N} k_n^{\alpha} \Big(\int_{t_{n-1}}^{t_n}\| \eta'\|_1dt\Big)^2\,.
\end{multline*}\end{theorem}
\begin{proof}
 This bound follows from the decomposition of the error  in (\ref{eq: decompose U-u}), the triangle inequality, Theorem
 \ref{lem: bound of theta},
and the fact (by the interpolation properties of the  operator $\Pi$),
$  \|\eta\|_{L_\infty(L_2)} = \max_{1\le n\le N} \Big(\sup_{t\in I_n} \|\eta(t)\|\Big) \le  \max_{1\le n\le N} \int_{t_{n-1}}^{t_n}\| \eta'\|dt\,.$ $\quad \Box$
\end{proof}
\section{$h$-version errors} \label{sec:algebraic} In this section, we  focus on the explicit error bounds of the $h$-DG solution  $U$ of uniform degree $p$ on each subinterval $I_n$ for $2\le n\le N.$  Because of the singular behavior of the  solution  $u$ of \eqref{eq:ivp} near $t=0,$ the degree of $U$  on the first  subinterval $I_1$ will be chosen to be one (i.e., $p_1=1$). So,  ${\bf p}=(1,p,\cdots,p)$.
  However, the numerical results suggested that this modification is not always needed. More precisely,
 we are required to consider $p_1=1$ if the time mesh, \eqref{eq: standard tn}, is strongly graded.

Following~\cite{McLeanMustapha2009,MustaphaMcLean2011,MustaphaMcLean2012IMA}, we assume that the solution $u$ of~(\ref{eq:ivp}) satisfies:
\begin{equation} \label{eq:countable-regularity v1}
\|u^{(j)}(t)\|_1\le M\, t^{\sigma-j}\qquad \forall\, 1\le j\le p+1,
\end{equation}
for some positive constants $M$ and $\sigma$; for a proof we refer the reader to~\cite{McLean2010,McLeanMustapha2007}\,.

To compensate for singular behaviour of $u$ near $t=0$, we employ a family of non-uniform meshes denoted by ${\cal M}_\gamma$, where the
time-steps are graded towards $t=0$. Following \cite{McLeanMustapha2009,MustaphaMcLean2011,MustaphaMcLean2012IMA,MustaphaMcLean2012SINUM}, for a fixed parameter $\gamma\ge1$,  we assume that
\begin{equation}\label{eq: standard tn}
t_n=(nk)^\gamma \quad{\rm with}~~ k=\frac{T^{1/\gamma}}{N}\quad \text{for ~~$0\le n\le N$.}
\end{equation}
Noting that the time step sizes are nondecreasing, that is, $k_i\le k_j$ for $i\le j$. Moreover, one can show that
\begin{equation}\label{eq: mesh assumption} k_n\le \gamma k t_n^{1-1/\gamma}~~{\rm for}~~n\ge 1 ~~{\rm and}~~
t_n\le 2^\gamma t_{n-1}~{\rm for}~~ n\ge 2\,.\end{equation}

In the next theorem, we derive the error estimate for the $h$-DG solution over the graded mesh $\cal M_\gamma$\,. In the
$L_\infty(0,T)$-norm,  we prove an $O(k^{\max\{2,p+\frac{1}{2}\}+\frac{\alpha}{2}})$ convergence rate, i.e., short by $-\frac{\alpha}{2}$ power from
being optimal for $p= 1$ and by $\frac{1-\alpha}{2}$ power  for $p\ge 2$. However, the numerical results  indicate optimal $O(k^{p+1})$-rates
for  $p\ge 1$. Indeed, these results are high-order extensions (also improvements) of the ones shown in  \cite{McLeanMustapha2009,MustaphaMcLean2011,MustaphaMcLean2012IMA} for $p\in \{0, 1\}.$ In contrast, for $0<\alpha<1$, we successfully proved optimal $O(k^{p+1})$ convergence rates in \cite[Theorem 4.9]{MustaphaSchoetzau2013}, but in a much weaker norm. Noting that, the proof here is more technical but the general approach is partially similar to the proof of Theorem 4.9 in  \cite{MustaphaSchoetzau2013}.
\begin{theorem}\label{thm: ||U-u||}
Let the solution $u$ of~\eqref{eq:ivp} satisfy the regularity property~\eqref{eq:countable-regularity v1}. Let $U\in \W({\cal M}_{\gamma},{\bf
p})$ be the $h$-DG approximation with $U^0_-=u_0$. Then for $\gamma\ge1$, we have
\[
 \|U-u\|_{L_\infty(L_2)} \le C\,\times \begin{cases}
    k^{\min\{\gamma(\sigma+\frac{\alpha}{2}),2+\frac{\alpha}{2}\}}&\quad{\rm for}\quad
  p =   1,\\
  k^{\min\{\gamma(\sigma+\frac{\alpha}{2}),p+1+\frac{\alpha}{2}\}-\frac{1}{2}}&\quad{\rm for}\quad
  p \ge   2\end{cases}
\] where $C$ is a constant that depends only on $T$, $\alpha$,  $\gamma$, $\sigma$
 and $p$.
\end{theorem}
\emph{Proof} Theorem~\ref{thm:error-bound-abstract} yields
\begin{multline*}
 \min\{1,\min_{n=2}^{N}\big(t_n^{-1}k_n^{\min\{1,p-1\}}\big)\} \|u-U\|_{L_\infty(L_2)}^2
  \\ \le C \,k_1^{\alpha} \left(\int_0^{t_1}\| \eta'\|_1\right)^2+ C \max_{n=2}^N
k_n^{\alpha+1} \int_{t_{n-1}}^{t_n}\|\eta'\|_1^2.\end{multline*}
 Since
\[t_n^{-1}k_n^{\min\{1,p-1\}}=
1-\frac{t_{n-1}}{t_n}=1-\left(1-1/n\right)^\gamma \ge 1/n \ge 1/N~{\rm for}~p\ge 2,
\]
\[  \frac{1}{N^{\min\{1,p-1\}}}\|u-U\|_{L_\infty(L_2)}^2 \le C \,k_1^{\alpha} \left(\int_0^{t_1}\| \eta'\|_1dt\right)^2+ C \max_{n=2}^N
k_n^{\alpha+1} \int_{t_{n-1}}^{t_n}\|\eta'\|_1^2dt.\] On the  subinterval $I_1$, $\Pi u \in \Poly_1$ and satisfies:
\[
\Pi u(t_1^-)=u(t_1)\quad\text{and}\quad \int_0^{t_1} \bigl[u(t)-\Pi u(t)]\,dt=0\,.
\]
 Explicitly,
the derivative of the interpolation error admits the integral representations~\cite[Equation~(3.8)]{MustaphaMcLean2009}:
\begin{equation}\label{eq: bound of the projection new}
\begin{aligned}
\eta'(t)&=-u'(t)+\frac{2}{k_1^2}\int_0^{t_1}
    s\,u'(s)\,ds,
        \quad\text{for $t\in I_1$.}
\end{aligned}
\end{equation}
So,  from the triangle inequality and \eqref{eq: bound of the projection new}, we notice that
      \[\int_0^{t_1}\,\| \eta '\|_1\,dt \le \int_0^{t_1}\left(\|u'(t)\|_1+\frac{2}{k_1}\,\int_0^{t_1}
    \|u'(s)\|_1\,ds\right)\,dt\le 3\int_0^{t_1}
    \|u'(t)\|_1\,dt\,.\]
Thus,   using
 the regularity assumption,  \eqref{eq:countable-regularity v1}, and the mesh property, \eqref{eq: mesh assumption},
\begin{equation}
\label{eq:exp-I1-L2 new new}k_1^{\alpha}\left(\int_0^{t_1}\|\eta'\|_1dt\right)^2 \le C\,k_1^{\alpha} \Bigl(\int_0^{t_1} t^{\sigma-1}\,dt\Bigr)^2
     =  C\frac{k_1^{2\sigma+\alpha}}{\sigma}  \le C k^{\gamma(2\sigma+\alpha)}\quad {\rm for}~~ \gamma \ge 1\,.
\end{equation}
In addition, for $n\ge2$, we use Theorem~\ref{thm:eta-approx} and get
\begin{equation*}
\begin{aligned}
k_n^{\alpha+1} \int_{t_{n-1}}^{t_n}\|\eta'\|_1^2dt &\le C  k_n^{2p+\alpha+1}
\int_{t_{n-1}}^{t_n}\|u^{(p+1)}\|^2dt\\
   & \le C\, k_n^{2p+2+\alpha}{t_n}^{2(\sigma-1-p)}\\
   & \le C\,k^{2p+2+\alpha} t_n^{2\sigma+\alpha -(2p+2+\alpha)/\gamma}\\
&       \le C\,k^{\min\{\gamma(2\sigma+\alpha),2p+2+\alpha\} }\quad {\rm for}~~  \gamma \ge 1.\quad\Box
\end{aligned}
\end{equation*}
\section{$hp$-version errors}\label{sec:exponential}
We discuss the error results of the $hp$-DGM based on geometrically refined time-steps and linearly increasing approximation orders.
Following~\cite{MustaphaBrunnerMustaphaSchoetzau2011,MustaphaSchoetzau2013}, we consider the $hp$-DGM for problems with solutions that have
start-up singularities as $t\to 0$, but are analytic for $t>0$. More precisely, we stipulate that the solution $u$ of \eqref{eq:ivp} has the
analytic regularity:
\begin{equation}
\label{eq:countable-regularity} \|u^{(j)}(t)\|+t\|u^{(j)}(t)\|_1\le M\, d^j \Gamma(j+1) t^{\sigma-j}\qquad \forall\, t \in (0,T],\ \forall\, j
\geq 1,
\end{equation}
for positive constants $\sigma$, $M$ and $d$. Proving the regularity statement~\eqref{eq:countable-regularity} remains an open issue, which is
beyond the scope of the present paper.

To resolve the singular behavior of the solution near $t=0$, we shall make use of geometrically refined time-steps and linearly increasing
degree vectors, and apply the $hp$-techniques that were developed
in~\cite{MustaphaBrunnerMustaphaSchoetzau2011,MustaphaSchoetzau2013,SchoetzauSchwab00}. To describe this, we first partition $(0,T)$ into
(coarse) time intervals $\{\mathfrak J_i\}_{i=1}^{\bf K}$. The first interval $\mathfrak J_1=(0,T_1)$ is then further subdivided geometrically
into $L+1$ subintervals $\{I_n\}_{n=1}^{L+1}$ as follows:
\begin{equation}
\label{eq:geometric-grading} t_0=0,\qquad t_n=\delta^{L+1-n} T_1\quad\text{for $1\leq n\leq L+1$}.
\end{equation}
As usual, we call $\delta\in(0,1)$ the geometric refinement factor, and $L$ is the number of refinement levels. From
(\ref{eq:geometric-grading}), we observe that the subintervals $\{I_n\}_{n=1}^{L+1}$ satisfy
\begin{equation}
\label{eq:lambda} k_n=t_n-t_{n-1}=\lambda t_{n-1}\quad {\rm with}\quad \lambda=(1-\delta)/\delta\quad{\rm for}~~n\ge 2.
\end{equation}
Let ${\cal M}_{L,\delta}$ be the mesh on $(0,T)$ defined in this way. The polynomial degree distribution $\bf p$ on ${\cal M}_{L,\delta}$ is
defined as follows. On the first coarse interval $\mathfrak J_1$ the degrees are chosen to be linearly increasing:
\begin{equation}
\label{eq:mu-linear} p_n=\lfloor \mu n \rfloor\qquad\text{for $1\leq n\leq L+1$},
\end{equation}
for a slope parameter $\mu>0$. On the coarse time intervals $\{\mathfrak J_i\}_{i=2}^{\bf K}$ away from $t=0$, we set the approximation degrees
uniformly to $p_{L+1}=\lfloor \mu (L+1) \rfloor$. The resulting $hp$-version finite element space is denoted by $\W({\cal M}_{L,\delta},\bf p)$.

Our main result of this section suggests  that non-smooth solutions satisfying~\eqref{eq:countable-regularity} can be approximated at
exponential rates convergence on the $hp$-version discretizations introduced above.  This will be done by proceeding along the lines of \cite[Theorem 4.2]{MustaphaBrunnerMustaphaSchoetzau2011} in our earlier work.
\begin{theorem}\label{thm:exponential-convergence}
Let $U\in \W({\cal
M}_{L,\delta},{\bf p})$ be the $hp$-DG approximation  with $U_-^0=u_0$. Then there exists a slope $\mu_0>0$ depending on $\delta$ and the
constants $\sigma$ and $d$ in~\eqref{eq:countable-regularity} such that for linearly increasing polynomial degree vectors ${\bf p}$ with slope
$\mu\geq \mu_0$, 
\[
   \|U-u\|_{L_\infty(L_2)}\le C\,{\rm exp}(-b\sqrt{\N}),
\]
with positive constants $C$ and $b$ that are independent of ~$\N:=\dim(\W({\cal M}_{L,\delta},{\bf p}))$, but depending on the problem
parameters $T$ and $\alpha$, the regularity parameters $M$, $d$ and $\sigma$ in~\eqref{eq:countable-regularity}, and the mesh parameters
$\delta$, $T_1$ and $\mu$.
\end{theorem}
\begin{proof}
From the geometric mesh assumptions \eqref{eq:geometric-grading}--\eqref{eq:lambda}, we notice that $t_i^{-1} k_i=1-\delta$ for $1\le i\le L+1$.
Hence, using  Theorem~\ref{thm:error-bound-abstract} and obtain
\begin{equation}
\label{eq:exp-lbl10}   \|U-u\|^2_{L_\infty(L_2)}\le C p_{L+1}^2\max\{1/(1-\delta),{\bf K}\} (E_{1}+E_2),
\end{equation}
where
\[\begin{split}
 E_1&= k_1^{\alpha}  \left(\int_{I_1}\|\eta'\|_1dt\right)^2+\max_{i=2}^{L+1}k_i^{\alpha+1}\int_{t_{i-1}}^{t_i}\| \eta'\|_1^2 dt
\\
 E_2&= \max_{i=2}^{K}(T_i-T_{i-1})^{\alpha+1} \,\int_{T_{i-1}}^{T_i}\| \eta'\|_1^2dt\,.
\end{split}
\] Since the solution $u$ is analytic on the coarse elements ${\mathfrak J}_i$, $2\leq i\leq {\bf K}$, from Theorem~\ref{thm:eta-approx} and the
approximation results for analytic functions in~\cite[Theorem~3.19]{Schwab98} yields an error estimate of the form
\begin{equation}
\label{eq:exp-lbl20} E_2 \leq C_1\, {\rm exp}(-b_1 L).
\end{equation}
On the first subinterval $I_1$ adjacent to $t=0$, $\Pi u \in \Poly_1$. Hence, we
 follow the steps in \eqref{eq:exp-I1-L2 new new}, and then  using the regularity assumption (\ref{eq:countable-regularity}) and the geometric mesh properties, \eqref{eq:geometric-grading},
\begin{equation}
\label{eq:exp-I1-L2 new}k_1^{\alpha}\left(\int_0^{t_1}\|\eta'\|_1dt\right)^2 \le C\,k_1^{\alpha} \Bigl(\int_0^{t_1} t^{\sigma-1}\,dt\Bigr)^2
     =  C\frac{k_1^{2\sigma+\alpha}}{\sigma}  \le C_2 {\rm exp}(-b_2 L).
\end{equation}
On the subintervals $I_j$ for $2\le j\le L+1$,
 from  the regularity property
(\ref{eq:countable-regularity}), we readily conclude that
 \begin{align*}
 \int_{t_{j-1}}^{t_j} \|u^{(q_j+1)}\|_1^2dt
 &\le C \,d^{2q_j}\Gamma(q_j+1)^2
\int_{t_{j-1}}^{t_j}\,t^{2(\sigma-1-q_j)}\,dt\\
 &\le C \,k_j\,d^{2q_j}\Gamma(q_j+1)^2
    t_{j-1}^{2(\sigma-1-q_j)},
\end{align*}
and hence, we use Theorem~\ref{thm:eta-approx} and the equality $k_j=\lambda\,t_{j-1}$ with $t_{j-1}\le \delta^{L+2-j}T_1$
(from~\eqref{eq:lambda} and \eqref{eq:geometric-grading}), and  get
 \begin{equation*}
 \begin{split}
k_j^{\alpha+1} \int_{t_{j-1}}^{t_j}\|\eta'\|_1^2dt &  \le C  \,\Gamma_{p_j, q_j}\,p_j^2\left(\frac{k_j}{2}\right)^{2(q_j+1)+\alpha}d^{2q_j}\Gamma(q_j+1)^2
    t_{j-1}^{2(\sigma-1-q_j)}\\
     & \le C \,\Gamma_{p_j, q_j}\,p_j^2\,
\left(\frac{d\lambda}{2}\right)^{2q_j}\Gamma(q_j+1)^2
   \delta^{(2\sigma+\alpha)(L+2-j)}\,.
\end{split}
\end{equation*}

Using interpolation arguments analogous to~\cite[Lemma~3.39]{Schwab98}, it can be seen that the above inequality also holds for any non-integer
regularity parameter $q_j$ with $0\leq q_j\leq p_j$. Thus, we take $q_j=c_jp_j$ with $c_j\in(0,1)$ and proceed as in~\cite[Theorem
3.36]{Schwab98},  and obtain
\begin{equation*}
\Gamma_{p_j, q_j}\,\left(\frac{d\lambda}{2}\right)^{2q_j}\Gamma(q_j+1)^2 \le C p_j\left(\left(\frac{\lambda d
c_j}{2}\right)^{2c_j}\frac{(1-c_j)^{1-c_j}}{(1+c_j)^{1+c_j}}\right)^{p_j}.
\end{equation*}
Noting that
\[
\underset{0<c_j<1}{\inf}\left(\frac{\lambda\,dc_j}{2}\right)^{2c_j}\frac{(1-c_j)^{1-c_j}}{(1+c_j)^{1+c_j}}
 =:\ell_{\lambda,d}(c_{\min})<1\quad {\rm with}~~ c_{\min}=\frac{1}{\sqrt{1+(\lambda d/2)^2}}, \]
and consequently, choosing  $c_j=c_{\min}$ and  $p_j=\lfloor \mu j \rfloor\geq \mu_0 j$ with $\mu_0>0$ such that
$\left(\ell_{\lambda,d}(c_{\min})\right)^{\mu_0}=\delta^{(2\sigma+\alpha)},$  we conclude that
\begin{equation}
\label{eq:exp-In-L2}  \begin{aligned}
 k_j^{\alpha+1} \int_{t_{j-1}}^{t_j}\|\eta'\|_1^2dt  &\le C\delta^{(2\sigma+\alpha) L}p_{L+1}^3\,
(\ell_{\lambda,d}(c_{\min}))^{p_j}
   \delta^{-(2\sigma+\alpha) j }\\
   &\le  C\delta^{(2\sigma+\alpha) L}p_{L+1}^3\leq C_3 {\rm exp}(-b_3
L).\end{aligned}
\end{equation}
where we have absorbed the term $p_{L+1}^3$ into the constants~$C_3$ and~$b_3$.

Finally, by referring to~\eqref{eq:exp-lbl10},  \eqref{eq:exp-lbl20}, \eqref{eq:exp-I1-L2 new} and \eqref{eq:exp-In-L2} yields \[
\|U-u\|^2_{L_\infty(L_2)}\le C_4 \,{\rm exp}(-b_4 L), \]where we have absorbed the term $p_{L+1}^2\max\{1/(1-\delta),{\bf K}\}$ in \eqref{eq:exp-lbl10}
into the constant $C_4$ and $b_4$. Finally, since $\N=\dim(\W({\cal M}_{L,\delta},{\bf p})) \le CL^2$ for $L$ sufficiently large, we obtain the
desired result.$\quad \Box$
\end{proof}

\section{Numerical results}
\label{sec:numerics}
 In this section, we demonstrate the validity of the achieved error estimates for both the h-DG  and hp-DG
  time-stepping schemes,
  for problems of the form  \eqref{eq:ivp} when  $Au=-u_{xx}$ and   $\Omega=(0,1)$. To compute  our numerical solution,
 we  discretize in space using   the standard FEs. So, we construct a family of uniform partitions of the domain~$\Omega$
into subintervals with step size $h$, and let $S_h\subset H_0^1(\Omega)$ denote the space of continuous, piecewise polynomial functions of
degree $\le r$ with $r \ge 1$. The discontinuous finite element space~\eqref{eq:FE-space} is now modified to the fully discrete finite
dimensional space
\begin{equation}
\label{eq:trialspace} \W({\cal M},{\bf p},S_h)=\left\{\,U_h:[0,T]\to S_h\,:\, U_h|_{I_n}\in \Poly_{p_n}(S_h), \ 1\leq n\leq N\,\right\}
\end{equation}
where by $\Poly_{p}(S_h)$ we denote the space of polynomials of degree~$\le p$ in the time variable with coefficients in~$S_h$.

We define our fully-discrete time-stepping DG-spatial FE scheme as follows: find $U_h\in\W({\cal M},{\bf p},S_h)$ such that
\begin{equation}\label{eq: GN Uh}
\begin{aligned}
G_N(U_h,X)&=\iprod{R_h u_0,X^0_+}+\int_0^{t_N}\,\iprod{f(t),X(t)}\,dt
    \qquad\forall\, X\in\W({\cal M},{\bf p},S_h)
\end{aligned}
\end{equation}
 where $G_N$ is the global bilinear form defined as in~\eqref{eq: GN def} and  $R_h:H_0^1(\Omega)\to S_h$ is the Ritz projection given by
  $A(R_hv,\chi)=A(v,\chi)$ for all $\chi\in S_h$\,.

To demonstrate the validity of the algebraic and exponential  convergence results of Theorems \ref{thm: ||U-u||} and
~\ref{thm:exponential-convergence} for the fully discrete version scheme, we choose $h$ (the spatial step size) and $r$ (the degree of the
approximate FE solution in the spatial variable) so that the temporal errors are dominating. To evaluate the errors, we introduce the finer grid
\begin{equation}\label{eq: fine grid}
\G^{m}=\{\,t_{j-1}+ n k_j/m\,:\, 1\leq j \leq N,\ 0 \leq n\leq m\,\}
\end{equation}
($N$ is the number of time mesh subintervals). Thus, for large values of~$m$, the error measure $ |||v|||_{m}:=\max_{t\in\G^{m}}\|v(t)\|$
approximate  the norm $\|v\|_{L_\infty(L_2)}$. To compute the spatial $L_2$-norm, we apply a composite Gauss quadrature rule with $(r+1)$ points
on each interval of the finest spatial mesh.

{\bf Example:} We choose the initial datum such that the exact solution is:
\begin{equation}\label{eq: num ex1}
 u(x,t) = {\sin}(\pi x)-t^{\alpha+2}{\sin}(2\pi x).
\end{equation}
\begin{table}[htb]
\renewcommand{\arraystretch}{1.0}
\begin{center}
\begin{tabular}{|r|rr|rr|rr|rr|r|}
\hline $N$ &\multicolumn{2}{c|}{$\gamma=1$} &\multicolumn{2}{c|}{$\gamma=1.3$}&\multicolumn{2}{c|}{$\gamma=1.6$}&\multicolumn{2}{c|}{}
&\multicolumn{1}{c|}{}
\\ \hline
 18& 8.32e-04&     & 2.78e-04&     & 1.93e-04&     & & &\\
 27& 4.80e-04& 1.35& 1.36e-04& 1.76& 8.28e-05& 2.08& & &$p=1$\\
 36& 3.27e-04& 1.34& 8.28e-05& 1.73& 4.59e-05& 2.05& & &\\
 72& 1.31e-04& 1.32& 2.53e-05& 1.71& 1.12e-05& 2.03& & &\\
 \hline
 \hline $N$ &\multicolumn{2}{c|}{$\gamma=1$} &\multicolumn{2}{c|}{$\gamma=1.6$}&\multicolumn{2}{c|}{$\gamma=2.3$}&\multicolumn{2}{c|}{$\gamma=3$}&\multicolumn{1}{c|}{}
\\ \hline
18& 1.07e-04&     & 1.18e-05&      & 2.64e-06&     & & &\\
27& 6.18e-05& 1.36& 4.87e-06&  2.18& 7.43e-07& 3.12& & &$p=2$\\
36& 4.20e-05& 1.34& 2.62e-06&  2.15& 3.06e-07& 3.08& & &\\
72& 1.67e-05& 1.33& 6.06e-07&  2.11& 4.13e-08& 2.89& & &\\
  \hline
\hline
  9 & 1.01e-04&     & 2.00e-05&     & 3.65e-06&     & 2.42e-06&     &\\
 18 & 3.81e-05& 1.41& 4.18e-06& 2.26& 3.87e-07& 3.23& 1.30e-07& 4.22&$p=3$\\
 27 & 2.19e-05& 1.36& 1.72e-06& 2.18& 1.10e-07& 3.10& 2.79e-08& 3.80&\\
 36 & 1.49e-05& 1.34& 9.29e-07& 2.15& 4.54e-08& 3.08& 1.03e-08& 3.46&\\
\hline
  \end{tabular} \vspace{0.05in} \caption {The errors $|||U_h-u|||_{10}$ for
the $h$-DGM for different mesh gradings  with $\alpha=-0.7$. We observe convergence of order $k^{(\alpha+2)\gamma} (=k^{1.3\gamma})$ for $1\le
\gamma \le (p+1)/(\alpha+2)$ for $p=1,\,2$ with some deterioration in the convergence rates for $p=3$ and $\gamma=3$. This might be due to the
direct implementation of the discrete solution which will then cause some numerical instability in computing the integrals involved the
 memory term especially when $p\ge 3.$ Indeed, for $p=3$ and $\gamma=3$,  modifying the order of the DG solution
 on the first time subinterval $I_1$ by replacing it with a linear DG approximation (as we assumed in the theory) was beneficial.} \label{tab: ||U_h-u|| alpha=-0.7 smooth}
\end{center}
\end{table}
It can be seen that the regularity conditions~\eqref{eq:countable-regularity} and~\eqref{eq:countable-regularity v1} hold for $\sigma=\alpha+2$.
We first test the accuracy of the h-DGM  with uniform polynomial degree $p$ (in time) on the non-uniformly graded meshes ${\cal M}={\cal
M}_\gamma$ in~\eqref{eq: standard tn} for various choices of~$\gamma\ge1$ and for $\alpha=-0.7$.
 In  Table~\ref{tab: ||U_h-u|| alpha=-0.7 smooth}  we computed  the errors and the experimental rates of convergence for various
values of $\gamma$.  We observe  a uniform global error  bounded by $Ck^{\min\{\gamma(2+\alpha),p+1\}}$ for $\gamma \ge 1$ (in particular for
$p\in\{1,2\}$), which is optimal for $\gamma\ge (p+1)/(\alpha+2).$  These numerical results  illustrated more optimistic convergence rates
(faster and optimal) compared to Theorem~\ref{thm: ||U-u||},
 and also demonstrated that the grading mesh parameter $\gamma$ is slightly relaxed.  Recall that, for a strongly graded mesh, the achieved convergence rate
 in Theorem~\ref{thm: ||U-u||} is  of order
 $ O(k^{\min\{\gamma(2+\frac{3\alpha}{2}),2+\frac{\alpha}{2}\}})$ for $p=1$ and
$O(k^{\min\{3\gamma\frac{(\alpha+1)}{2},p+\frac{\alpha+1}{2}\}})$ for $p\ge 2,$ i.e., short by order $-\frac{\alpha}{2}$
 from being optimal for $p=1$ while short by order $\frac{1-\alpha}{2}$ for $p\ge 2$ (more pessimistic)

Next, we test the performance of the $hp$-version time-stepping of the scheme \eqref{eq: GN Uh}. We use the geometrically refined time-step and
linearly increasing polynomial degrees as introduced in Section~\ref{sec:exponential} for the exact solution in~\eqref{eq: num ex1} with
$\alpha=-0.7$. We choose $T_1=1$ and $\mu=1$. We notice that the analytic regularity property  \eqref{eq:countable-regularity} holds for
$\sigma=\alpha+2$ and hence,  in accordance with Theorem~\ref{thm:exponential-convergence},  we expect the error to converge exponentially
(${\rm exp}(-b  \sqrt{\N})$ with $\N=\dim(\W({\cal M}_{L,\delta},{\bf p}))$). We calculate the coefficient $b$ in the exponent using the
formula:
\begin{equation}\label{eq: OC hp}
{\rm log}({\rm error}(\N_{L-1})/{\rm error}(\N_L))/({\sqrt{\N_L}-\sqrt{\N_{L-1}}}),
\end{equation}
where $\N_L=\dim(\W({\cal M}_{L,\delta},{\bf p}))$ and ${\rm error}(\N_L)$ is the  error in $L_\infty(0,T)$ corresponding to the geometric time
mesh \eqref{eq:geometric-grading} (with $T_1=T=1$) which consists of $L+1$ subintervals. The numerical values of~$b$ are approximately the same
(as it should be) for different values of geometric gradings~$L$.  This is illustrated tabularly in Table~\ref{tab: hp error v2 smooth} where it
can be seen that $\delta=0.24$, the $hp$-version gives an $L_\infty$-error smaller than $e^{-07}$ with less than $44$ degrees of freedom and $8$
time subintervals only.
 This clearly underlines the suitability of $hp$-version approaches for the numerical approximation of the fractional diffusion problem \eqref{eq:ivp}.
 We show the $hp$-errors against $\sqrt{\N}$ graphically  in Figure \ref{fig1: hp error smooth}.
 In the semi-logarithmic scale, the curves are roughly straight lines, which indicates exponential convergence rates.

\begin{table}
\renewcommand{\arraystretch}{1.0}
\begin{center}
\begin{tabular}{|r|r|rr|rr|rr|rr|}
\hline $L$&$\N(L)$ &\multicolumn{2}{c|}{$\delta=0.21$}& \multicolumn{2}{c|}{$\delta=0.24$}&\multicolumn{2}{c|}{$\delta=0.27$}
&\multicolumn{2}{c|}{$\delta=0.30$}\\
\hline
3& 14& 1.58e-04& 2.72  &2.66e-04&  2.49& 4.20e-04& 2.29& 6.33e-04& 2.10\\
4& 20& 2.11e-05& 2.76  &4.20e-05&  2.53& 7.72e-05& 2.32& 1.33e-04& 2.13\\
5& 27& 3.76e-06& 2.38  &6.65e-06&  2.55& 1.42e-05& 2.34& 2.81e-05& 2.15\\
6& 35& 1.09e-06& 1.72  &1.06e-06&  2.55& 2.63e-06& 2.35& 5.93e-06& 2.16\\
7& 44& 1.01e-06& 0.09  &2.49e-07&  2.02& 4.86e-07& 2.35& 1.26e-06& 2.16\\
 \hline
  \end{tabular}\vspace{0.05in}
   \caption {The errors
$|||U_h-u|||_{60}$ and the calculated exponent $b$ for different choices of~$\delta$ with  $\alpha=-0.7$. We partitioned the time interval
geometrically (see \eqref{eq:geometric-grading}) into $L+1$ subintervals. } \label{tab: hp error v2 smooth}
\end{center}
\end{table}
 \begin{figure}
 \begin{center}
 \scalebox{0.4}{\includegraphics{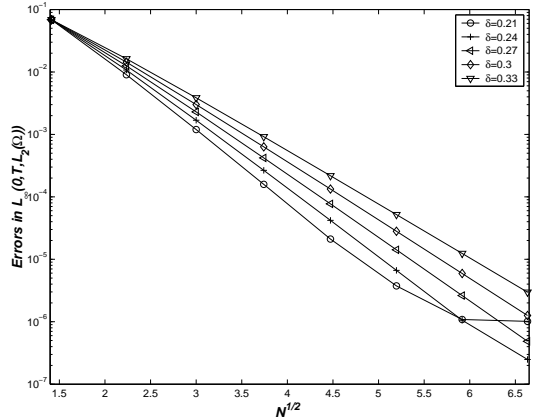}}
  \caption{The errors $|||U_h-u|||_{60}$  plotted against $\sqrt{\N}$ for different choices of~$\delta$,
 with $\alpha=-0.7$.}\label{fig1: hp error smooth}
 \end{center}
 \end{figure}

 In Figure \ref{fig1: hp error smooth different delta}, for a fixed $\N=44$,   we plot
 the errors against the parameter $\delta$ for different values of $\alpha$. We observe that  values of $\delta$ in the neighborhood of the interval
$[0.2, 0.3]$ yields the best results.
 \begin{figure}[htb]
 \begin{center}
 \scalebox{0.5}{\includegraphics{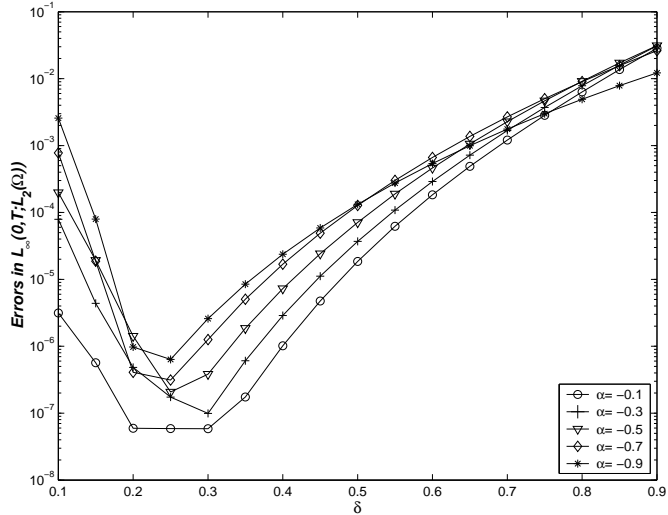}}
 \caption{The errors $|||U_h-u|||_{60}$ plotted against  $\delta$
for different values of  $\alpha$ and fixed $\N=44$.} \label{fig1: hp error smooth different delta}
 \end{center}
 \end{figure}

\bibliographystyle{plain}

\end{document}